\documentclass[12pt]{article}
\usepackage[english]{babel}
\usepackage{geometry} 
\usepackage{color}

\usepackage{latexsym} 
\usepackage{amsmath,amssymb,amsthm}
\usepackage[applemac]{inputenc}

 \newtheorem{thm}{Theorem}[section]
 
 \newtheorem{lem}[thm]{Lemma}
 \newtheorem{prop}[thm]{Proposition}
 \theoremstyle{definition}
 \newtheorem{definition}[thm]{Definition}

 \numberwithin{equation}{section}

\newcommand{\half}{{\textstyle{1\over2}}}
\newcommand{\R}{\mathbb R}%
\newcommand*{\C}{\mathbb C}%

\newcommand{\g}{\mathfrak{g}}

\newcommand{\scal}{{s}}
\newcommand{\Cric}{{ChRic}}

\makeatletter
\def\operatorname#1{\mathop{\operator@font #1}\nolimits}%
\makeatother

\DeclareMathOperator{\End}{End}
\DeclareMathOperator{\Tr}{Tr}

\DeclareMathOperator{\Ker}{Ker}
\newcommand{\Span}{\operatorname{Span}}
\newcommand{\Image}{\operatorname{Im}}

\DeclareMathOperator{\Id}{Id}

\newcommand*{\cyclic}{\mathop{\kern0.9ex{{+}\kern-2.2ex\raise-.29ex%
      \hbox{\Large\hbox{$\circlearrowright$}}}}\limits}

\title{Distributions associated to  almost complex structures 
on  symplectic manifolds 
. 
}

\author{
Michel Cahen$^{(1)}$, Maxime G\'erard$^{(2)}$,  Simone Gutt$^{(1,2)}$,  and
 Manar Hayyani$^{(3)}$\\
\scriptsize{mcahen@ulb.ac.be,  maxime.gerard57@gmail.com, 
 sgutt@ulb.ac.be, manar.hayyani@yahoo.fr}\\
\footnotesize{(1)  D\'{e}partement de Math\'{e}matique, Universit\'{e} Libre de Bruxelles}\\[-7pt]
\footnotesize Campus Plaine, CP 218, Boulevard du Triomphe\\[-7pt]
\footnotesize BE -- 1050 Bruxelles, Belgium\\[-7pt]
\footnotesize Membre de l'Acad\'emie Royale de Belgique.\\
\footnotesize (2) Universit\'e de Lorraine\\ [-7pt]
\footnotesize Institut Elie Cartan de Lorraine, UMR 7502,\\[-7pt]
\footnotesize Ile du Saulcy, F-57045 Metz, France. \\
\footnotesize (3) Universit\'e Moulay Isma\"\i l,\\ [-7pt]
\footnotesize  Mekn\`es, Maroc
}

\date{~\\
}

\date{}

\begin{document}

\maketitle

\begin{abstract}
We look at methods to select triples $(M,\omega,J)$ consisting of a symplectic manifold $(M,\omega)$ endowed with a compatible positive  almost complex structure $J$, in terms of the Nijenhuis  tensor $N^J$ associated to $J$.\\
We study in particular the image  distribution $\Image N^J$. 

 \end{abstract}

\section*{Introduction}
On a smooth symplectic manifold $(M,\omega)$ (whose dimension will be denoted $2n$ in this paper), an almost complex structure $J$ (i.e. a smooth field of endomorphisms of the tangent bundle whose square is equal to minus the identity) is said to be {\it{ compatible}} with $\omega$ if  the tensor $g_{J}$ defined by 
$
g_{J}(X,Y):=\omega(X,JY) 
$
is symmetric, hence yields an {\it{associated  pseudo-Riemannian metric}}.
A compatible almost complex structure is said to be {\it{positive}} when the associated metric $g_{J}$ is Riemannian.\\
It is well known 
that on any symplectic manifold, there exist positive compatible almost complex structures and that the space $\mathcal{J}(M,\omega)$ of those structures is an infinite dimensional contractible Fr\'echet space. 
This contractibility shows that  the Chern
classes  of a symplectic manifold $c_k(M,\omega):=c_k(TM,J)\in H^{2k}_{deRham}(M;\R)$ are well defined.
The choice of a generic $J$ in this space allowed Gromov and Floer to
define the notion of pseudo-holomorphic curves and their generalisations which are essential tools in the study of symplectic invariants and symplectic topology.
Here, instead of dealing with generic $J$'s, we review methods to select  particular $J$'s on a given symplectic manifold or more  precisely to select  triples $(M,\omega, J)$  with $(M,\omega)$ symplectic and $J$ a positive compatible almost complex structure.
Such a triple is equivalent to the data of an {\it{almost K\"ahler manifold}} $(M,g,J)$, i.e. an almost  Hermitian manifold 
(which is  a Riemannian manifold $(M,g)$ with an almost complex structure $J$  which is compatible in the sense that the tensor $\omega$ defined by $\omega(X,Y)=g(JX,Y)$ is skewsymmetric),  with the extra condition  that $d\omega=0$.\\

An almost complex structure $J$  on  a manifold $M$ is said to be {\it{integrable}} if it is induced by a complex structure on $M$; this means that one can locally define complex coordinates on $M$ and that the changes of coordinates are holomorphic; the associated almost complex structure is then given by $J \frac{\partial}{\partial x^k}= \frac{\partial}{\partial y^k}$ if $x^k$ and $y^k$ are the real and imaginary part of the local complex coordinates $\{z^k=x^k+iy^k \,\vert\, k=1,\ldots,n\}$.\\
The Newlander-Nirenberg theorem asserts that an almost complex structure $J$ on a manifold $M$  is integrable if and only if its Nijenhuis tensor $N^J$ vanishes identically.
Recall that the Nijenhuis tensor (also called Nijenhuis torsion) associated to an almost complex structure $J$ is the tensor of type 
 $(1,2)$ defined by
\begin{equation}
N^J(X,Y):=[JX,JY]-J[JX,Y]-J[X,JY]-[X,Y] \quad \forall X,Y\in \chi(M).
\end{equation}

A first natural selection principle for triples $(M,\omega,J)$ is to ask that $N^J=0$, hence to look at the class of manifolds which are both symplectic and complex, with compatibility conditions between these two structures;  the triple $(M,\omega,J)$ satisfies $N^J=0$ if and only if the manifold $(M,g_J,J)$ is {\it{K\"ahler}}, i.e. the covariant derivative of $J$, with respect to the Levi Civita connection associated to the metric $g_J$, vanishes.
Indeed, if $(M,g,J)$  is an almost Hermitian manifold  and  if $\nabla^g$ denotes the Levi Civita connection associated to the metric $g$,  one has:
\begin{equation}\label{eq:dJN}
2\omega((\nabla^g_XJ)Y,Z)=-d\omega(X,JY,Z)-d\omega(X,Y,JZ)+\omega(N^J(Y,Z),JX).
\end{equation}
K\"ahler manifolds form a very important class of symplectic manifolds. However there are many symplectic manifolds which do not admit a K\"ahler structure.\\
 A first compact example was given by Thurston in 1976, on the nilmanifold $M=\Gamma \backslash G$
 where $G$ is the $4$-dimensional nilpotent group
$G=\R^4=\R\times {\mathbb{H}}_3$, with multiplication defined  by
$(a^1,a^2,b^1,b^2)\cdot (x^1,x^2,y^1,y^2)=
 (a^1+x^1, a^2+x^2+b^1y^2,b^1+y^1,b^2+y^2)$, and $\Gamma=\mathbb{Z}^4$.
The group $G$ is endowed with the left invariant symplectic structure
$ \tilde\omega=dx^1\wedge dy^1+dx^2\wedge dy^2$,
  and  the symplectic structure on $M$ is the one lifting to $\tilde\omega$ on $G$.
We shall construct other examples of  such {\it{compact nilmanifolds}}, which are quotients of nilpotent Lie groups by lattices (i.e. discrete subgroups acting cocompactly).\\

To get ways of selecting  some geometrical data - here triples $(M,\omega, J)$ - we consider the three following interrelated methods:\\
- decompose into irreducible components  tensors associated to the data and 
impose conditions on some of those components;\\
- consider functionals defined from the geometrical input and look for extrema of those functionals;\\
- define distributions associated to the data and impose conditions on those.\\
In this paper, we apply those three approaches, using the Nijenhuis tensor $N^J$.\\

A possible method to select triples $(M,\omega,J)$  would be to ask for the annihilation of some components of the Nijenhuis tensor.
However the space of tensors  having the symmetries of the Nijenhuis tensor at a given point $x\in M$ is irreducible under the action of the unitary group viewed as  $Gl(T_xM,\omega_x,J_x):=Sp(V,\Omega)\cap Gl(V,j)$, the group of linear transformations of the tangent space at that point, $T_xM$, preserving the symplectic $2$-form $\omega_x$ and the almost complex structure $J_x$ at that point. We include a proof  of this result in section \ref{section:decompositionofN}.
Hence, the only meaningful linear condition one can impose to the Nijenhuis tensor is to vanish identically. This method thus leads only to   K\"ahler manifolds.\\
Observe , from equation \eqref{eq:dJN}, that the data of $\nabla^{g_J}J$ is equivalent to the data of  $N^J$. The irreducibility of $N^J$  shows in particular that there is no  notion of nearly K\"ahler symplectic manifold; recall that an almost Hermitian manifold $(M,g,J)$ is said to be {\it{nearly K\"ahler}} when $\left(\nabla^{g}_XJ\right)X=0$; if $d\omega=0$, nearly K\"ahler implies K\"ahler.\\

Other tensors can be associated to our data $(M,\omega,J)$ via the use of a connection. Natural connections, related tensors and some conditions on those which have been studied are reviewed in section \ref{section:connections}.\\

Blair and Ianus introduced in 1986 \cite{bib:BlairIanus} a variational principle based on a functional defined on the space $\mathcal{J}(M,\omega)$ of positive compatible almost complex structures on a compact symplectic manifold $(M,\omega)$, with the functional  defined in terms of the Nijenhuis tensor.  It associates to a $J\in \mathcal{J}(M,\omega) $  the integral of the square of the norm of the Nijenhuis tensor $N^J$:
$$
\mathcal{F}(J)=\int_M \left\Vert N^J \right\Vert^2_{g_J} \frac {\omega^n} {n!},
$$
where the norm on tensors is defined  using the metric $g_J$. \\
Since $N^J$ is irreducible under the action of the unitary group, all catenations of this tensor (using $g_J,\omega$ and their inverses) vanish; there is no linear function of $N^J$ (involving such catenations) which is invariant under the action of the unitary group. To construct a  functional involving $N^J$, the simplest is thus to integrate a polynomial of degree $2$ in $N^J$ involving catenations (using $g_J,\omega$ and their inverses). A straightforward calculation shows that up to a constant, there is only one such expression and it is the square of the norm of $N^J$.\\
This functional was  studied in 2012 by J.D. Evans in \cite{bib:JDEvans};
he showed that the infimum of this functional is zero, even when there is no K\"ahler structure, as soon as the symplectic $2$-form  $\omega$ is rational ($[\omega] \in H^2(M,\mathbb{Q})$).\\
Thurston's manifold $M=\mathbb{Z}^4 \backslash (\R\times {\mathbb{H}}_3)$, with its symplectic structure as described above, admits the $1$-parameter family of  positive compatible almost complex structures induced by  left invariant almost complex strutures  $J_{(\alpha)}$, for $\alpha>0$ on $\R\times {\mathbb{H}}_3$, 
whose matrix at the point $p=(x^1,x^2,y^1,y^2)$ in the basis $\{\frac{\partial}{\partial x^1}_{\vert_p},\frac{\partial}{\partial x^2}_{\vert_p},\frac{\partial}{\partial y^1}_{\vert_p},\frac{\partial}{\partial y^1}_{\vert_p}\}$
is given by {\scriptsize{$J_{(\alpha)_p}=\left(\begin{matrix} 0 & 0 &-1/\alpha &  0\\ 0 & y^1 & 0 &  -1-(y^1)^2\\ \alpha & 0 & 0 &  0 \\ 0 & 1 & 0 &  -y^1\end{matrix}\right)$}}.
Since the square of the norm of the Nijenhuis tensor  is the constant function equal to $\Vert N^{J_{(\alpha)}} \Vert^2_{g_{J_{(\alpha)}}}=  8\alpha$,
this example shows that a functional given by integrating any positive power of the norm of the Nijenhuis tensor would have the same defect.

Some other variational principles which have been considered are reviewed in section  \ref{section:variational principles}.\\

Given a triple $(M,\omega,J)$, we define the {\it{image distribution}}  on  $M$, denoted $\Image N^J$, whose value $(\Image N^J)_x$ at a point $x\in M $  is the subspace of the tangent space $T_xM$ spanned by all values  $N^J_x(X,Y)$. Since $N^J_x(JX,Y)=-JN^J_x(X,Y)$, this distribution is stable by $J$, hence symplectic. We also consider the distribution  $(\Image N^J)^\perp$ which associates to a point $x$  the orthogonal complement (which coincides whether it is relative to $\omega_x$ or to $g_{J_x}$) of  $(\Image N^J)_x$.\\
We prove in section \ref{section:dimdistrib}  that for any $4$-dimensional manifold, the dimension at any point $x$ of $(\Image N^J)_x$ is equal to $0$ or $2$.
In any other dimension $2n,\, n>2$, we give examples of triples $(M,\omega,J)$ for which  the dimension of $(\Image N^J)_x$ is constant and can take any even value between  $0$ and $2n$. Furthermore, all possible cases of involutivity or non involutivity for the distributions $\Image N^J$ and for $(\Image N^J)^\perp$ arise independently. The examples 
we give are  symplectic groups (i.e. Lie groups which carry a left invariant symplectic structure) with left invariant positive compatible almost complex structures; they
are described in  Maxime G\'erard's thesis \cite{bib:gerard}.\\



We say that  $J$ is {\it{maximally non integrable}} if $\Image N^J$ is the whole tangent bundle.
We prove in section \ref{section:parallelimage} that a triple $(M,\omega,J)$, for which  the image distribution is parallel with respect to the Levi Civita connection,  is necessarily locally the product of a K\"ahler manifold $(M_1,\omega_1,J_1)$ and a $(M_2,\omega_2,J_2)$ such that $J_2$ is maximally non integrable.

If the covariant derivative of the Nijenhuis tensor with respect to the Levi Civita connection  vanishes, it clearly implies that the image distribution is parallel. We prove the stronger result that given a triple $(M,\omega,J)$, then  $\nabla^{g_J} N^J=0$ implies $N^J=0$.\\

Section \ref{section:twistor} presents the two natural almost complex structures $J^{\pm}$ on the twistor space $\mathcal{T}$ over the hyperbolic space $H_{2n}$. This twistor space can be viewed as an  adjoint orbit in the semisimple Lie algebra $so(1,2n)$, and thus carries a Kirillov-Kostant-Souriau symplectic structure $\omega$ .
Both  $J^{\pm}$ are compatible with $\omega$, $J^-$ is not integrable and is positive, whereas $J^+$ is integrable but is not positive. We show that $J^-$ is maximally non integrable.

 \section*{Acknowledgement} This work is part of the project ``Symplectic techniques in differential geometry", funded by the ``Excellence of Science (EoS)" program 2018-2021 of the FWO/F.R.S-FRNS.

\section{Decomposition of the Nijenhuis tensor}\label{section:decompositionofN}

For any almost complex structure $J$, one has
\begin{equation}\label{prop:JNJ}
N^J(X,Y)=-N^J(Y,X),\qquad N^J(JX,Y)=-JN^J(X,Y),
\end{equation}
and if $J$ is compatible with a symplectic structure $\omega$, 
\begin{equation}\label{prop:cyclicN}
\cyclic_{X,Y,Z}\omega(N^J(X,Y),Z)=0
\end{equation}
where $\cyclic_{X,Y,Z}$ denotes the sum over cyclic permutations of $X,Y$ and $Z$.
{\scriptsize{Indeed 
\begin{eqnarray*}
\cyclic_{X,Y,Z}\omega(N^J(X,Y),Z)&=&\cyclic_{X,Y,Z}\left(\omega([JX,JY],Z)+\omega([JX,Y],JZ)+\omega([X,JY],JZ)-\omega([X,Y],Z)\right)\\
&=& \cyclic_{X,Y,Z}\left( -d\omega(JX,JY,Z)+JX\omega(JY,Z)+JY(\omega(Z,JX) +Z\omega(X,Y)-\omega([X,Y],Z)   \right)\\
&=& \cyclic_{X,Y,Z}\left( -d\omega(JX,JY,Z)+JX(\omega(JY,Z)+\omega(Y,JZ))+d\omega(X,Y,Z)\right)=0.
\end{eqnarray*}}}
As mentioned earlier, a first idea to select some positive compatible  almost complex structures on a symplectic manifold would be to try to annihilate some components of the corresponding $N^J$; for this to have a meaning, the set of components should be invariant under the natural pointwise action of the unitary group. This however is not possible (beyond the K\"ahler case, where one asks for the vanishing of the whole tensor $N^J$) since one has the following irreducibility property.
\begin{thm}
Consider a real symplectic vector space $(V,\Omega)$ of dimension $2n$ endowed with a positive compatible complex structure $j$. \\The space $\mathcal{N}(V,\Omega,j)$ of {\scriptsize{$\left(\begin{array}{c}1\cr 2\end{array}\right)$}}- tensors $T: V\times V\rightarrow V$ satisfying
$$T(X,Y)=-T(Y,X),\quad T(jX,Y)=-jT(X,Y), \quad {\textrm{and }}\, \cyclic_{X,Y,Z}\Omega(T(X,Y),Z)=0$$
is irreducible under the action of the unitary group $Sp(V,\Omega)\cap Gl(V,j)$.
\end{thm}
\begin{proof}
Denote by $W$ the space $V$ viewed as a complex vector space of dimension $n$, where the multiplication of a vector $X$ by $i$ is given by $jX$. Denote by $h$ the hermitian product on $W$ defined by $h(X,Y):=\Omega(X,jY)-i\Omega(X,Y)$;  the unitary group of $(W,h)$
is $U(W,h)=Sp(V,\Omega)\cap Gl(V,j)$.\\
For any $T\in \mathcal{N}(V,\Omega,j)$, define  $$\tilde{T} (X,Y,Z):=\Omega(T(X,Y),Z)-i\Omega(jT(X,Y),Z).$$ Since  $\tilde{T} (jX,Y,Z)= \tilde{T} (X,jY,Z)= \tilde{T} (X,Y,jZ)=i \tilde{T} (X,Y,Z)$, $\tilde{T}$ is a  $\C$-linear map $\tilde{T} : \otimes^3W\rightarrow \C$ with  the symmetry properties $$\tilde{T} (X,Y,Z)=-\tilde{T} (Y,X,Z)\qquad \textrm{ and } \qquad \cyclic_{X,Y,Z} \tilde{T}(X,Y,Z)=0.$$
The space $\widetilde{\mathcal{N}}$ of all $\C$-linear map $\otimes^3W\rightarrow \C$ with the above symmetries, is in bijection with $ \mathcal{N}(V,\Omega,j)$; it is  also the kernel of the map 
$a: \Lambda^2W^*\otimes W^*\rightarrow \Lambda^3W^*$ defined by skewsymmetrization.
Using Koszul long exact sequence, whose maps are given by  skewsymmetrization, $$0\rightarrow S^3W^*\rightarrow W^*\otimes S^2W^*\rightarrow \Lambda^2W^*\otimes W^*\rightarrow \Lambda^3W^*\rightarrow 0,$$ one sees that  the space $\widetilde{\mathcal{N}}$ is isomorphic to the image of $W^*\times S^2W^*$, hence to the quotient $$\widetilde{\mathcal{N}} \simeq(W^*\otimes S^2W^*)/S^3W^*.$$ To prove that the real  vector space $\mathcal{N}(V,\Omega,j)$ is irreducible under the action of  $Sp(V,\Omega)\cap Gl(V,j)$ is equivalent to show that  the complex space $\widetilde{\mathcal{N}}$ is irreducible under the action of  $U(W,h)$. This amounts to show that $W^*\otimes S^2W^*$ splits  exactly into two irreducible components under the action of $U(W,h)$.  This is equivalent to show that $W\times S^2W$ splits  exactly into two irreducible components under the action of $U(W,h)$, or, equivalently, 
 under the action of its Lie algebra $\mathfrak{u}(W,h)$,
  or  under the action of the  complexification $\mathfrak{gl}(W,\C)$ of $\mathfrak{u}(W,h)$, or, still equivalently,  under the action of the group $Gl(W,\C)$. Now any representation of $Gl(W,\C)$ on $\otimes^kW$ commutes with the action of the permutation group $S_k$ (acting by permutations of the factors) and this gives a splitting (see, for instance, Goodman and Wallach \cite{bib:GoodmanWallach})  $$\otimes^kW=\oplus_i V_i\otimes W_i$$ where the $V_i$'s are inequivalent irreducible representations of $Gl(W,\C)$ and the $W_i$'s are inequivalent representations of $S_k$. There exist exactly $3$ inequivalent representations of $S_3$, namely the trivial one on $\C$, the one on $\C$ given by the signature and the $2$-dimensional representation  corresponding to the action of $S_3$ on the space of vectors in $\C^3$ which are orthogonal to $(1,1,1)$. Hence the splitting of $\otimes^3W$ reads
 $$\otimes^3W= S^3W\oplus \Lambda^3W\oplus \left(\tilde{V}\otimes \C^2\right)$$ where $\tilde{V}$ is an irreducible representation of $Gl(W,\C)$. Since  $\otimes^3W=\left(W\otimes S^2W\right)\oplus \left(W\otimes \Lambda^2W\right)$,  where both factors are stable under the action, and since $S^3W$ (resp. $\Lambda^3W$) is a proper subset of $W\otimes S^2W$ (resp. $W\otimes \Lambda^2W$), 
 we have $W\otimes \Lambda^2W=\Lambda^3W \oplus \tilde{V}$ and $W\otimes S^2W=S^3W\oplus \tilde{V}$. The complex space $\widetilde{\mathcal{N}}$  is thus isomorphic to $\tilde{V}^*$ and is irreducible under the action of the unitary group.  Hence,  the real vector space $\mathcal{N}(V,\Omega,j)$ is indeed irreducible under the action of 
 $Sp(V,\Omega)\cap Gl(V,j)$.
\end{proof}

\section{Connections associated to a triple $(M\omega,J)$}\label{section:connections}
To select classes of triples $(M,\omega,J)$ via the use of associated tensors, a natural way is to choose a connection associated to the triple and look at its curvature.
We consider here the three most natural connections associated to a triple.\\

 {\bf{The Levi Civita connection $\nabla^{g_J}$}} is the canonical connection associated to the metric $g_J$; it is characterized by the fact that it is torsion free and $\nabla^{g_J}{g_J}=0$. 
 Remark that equation \eqref{eq:dJN} yields 
\begin{equation} \label{eq:NnablaJ}
2\omega((\nabla_X^{g_J}J)Y,Z)=\omega(N^J(Y,Z),JX)
\end{equation} 
so that the data of $\nabla_X^{g_J}J$ is equivalent to the data of $N^J$. In particular, the space of such tensors at a point is irreducible under the action of the unitary group.\\
 Since $N^J(JX,Y)=-JN^J(X,Y)$, equation \eqref{eq:NnablaJ} yields
\begin{equation} \label{eq:nablaJN}
\nabla_{JX}^{g_J}J=-J\nabla_{X}^{g_J}J.
\end{equation}

The curvature tensor $R^{g_J}$ of the Levi Civita connection can be decomposed into ten irreducible pieces under the action of the unitary group (Tricerri and Vanhecke  \cite{bib:Tricerri}).
The equation
$Ric^{g_J}(JX,JY)=Ric^{g_J}(X,Y)$ appears in a variational problem (see section \ref{section:variational principles}).
Let us remark  that neither $\omega$ nor $J$ are parallel with respect to $\nabla^{g_J}$ unless $(M,\omega,J)$ is K\"ahler. \\

 Any torsion free connection $\nabla$ on a symplectic manifold projects in a canonical way on a symplectic connection $\nabla'$ (i.e. a torsion free connection for which the symplectic form is parallel) defined by $\nabla'_XY:=\nabla_XY+\frac{1}{3}D(X,Y)+\frac{1}{3}D(Y,X)$ where $D$ is the tensor defined by $\omega(D(X,Y),Z)=(\nabla_X\omega)(Y,Z)$. Projecting the Levi Civita connection $\nabla^{g_J}$ in this way, since $(\nabla^{g_J}_X\omega)(Y,Z)=g((\nabla^{g_J}_XJ)Y,Z)=-\omega(J(\nabla^{g_J}_XJ)Y,Z)$,  one gets {\bf{the symplectic connection $\nabla^{s}$}} defined by
  \begin{equation}\label{eq:Symplconnection}
  \nabla^{s}_XY:=\nabla^{g_J}_XY-\frac{1}{3}J\left(\nabla^{g_J}_X J\right)Y-\frac{1}{3}J\left(\nabla^{g_J}_Y J\right)X.   \end{equation}
  The curvature tensor of such a connection splits, under the action of the symplectic group, into two irreducible components (Vaisman \cite{bib:Vaisman}). Its decomposition under the unitary group has not been achieved. 
 Remark  that $J$ is not  parallel with respect to $\nabla^{s}$ unless $(M,\omega,J)$ is K\"ahler.\\

A third natural connection associated to $(M,\omega,J)$  is {\bf{the Chern connection}} defined by  \begin{equation}\label{eq:Chernconnection}
\nabla^{C}_XY:=\nabla^{g_J}_XY-\frac{1}{2}J\left(\nabla^{g_J}_X J\right)Y;
\end{equation}
it is the unique linear connection $\nabla$ on $M$ such that  $\nabla\omega=0,\,\nabla J=0$ and such that its torsion is given by $T^\nabla(X,Y)=\frac{1}{4}N^J(X,Y)$ (equivalently $T^{\nabla}(JX,Y)=-JT^{\nabla}(X,Y)$). Torsion is unavoidable when the almost complex structure is not integrable : for any  connection $\nabla'$ with $\nabla' J=0$ one has 
\begin{equation}\label{eq:torsionN}
T^{\nabla'}(JX,JY)-JT^{\nabla'}(JX,Y)-JT^{\nabla'}(X,JY)-T^{\nabla'}(X,Y)=-N^J(X,Y).\end{equation} 
The curvature $R^{C}$ of this connection is a natural tensor associated to the data as well as tensors obtained by contractions of  this curvature, in particular   the {\it{Chern Ricci form }}
\begin{equation}\label{eq:ChernRicci}
\Cric(X,Y):=\Tr_\R JR^C(X,Y)=-2\Tr_\C R^C(X,Y),
\end{equation}
where $\Tr_\R$ denotes the real trace and where $\Tr_\C$ denotes the complex trace of the endomorphism viewed as complex linear on the $n$ dimensional complex vector space $(TM,J)$. The {\it{Hermitian scalar curvature}} 
$\scal^C=\scal^{C(\omega,J)}$ is defined by 
\begin{equation}\label{eq:Hermscalcurv}\scal^C\omega^n=n \,\Cric\wedge\omega^{n-1}.
\end{equation} Of course, all these can be expressed in terms of the Riemannian curvature $R^{g_J}$ and the Nijenhuis tensor $N^J$. One has 
    $\scal^C=\scal^{g_J}+2\Vert N^J\Vert_{g_J}^2,$ where, again,  the norm of the Nijenhuis tensor is the norm on tensors defined by extending the norm on $TM$ defined by $g_J$.
The Chern Ricci form  is  $2i$ times the curvature of the Hermitian connection induced by $\nabla^C$ on the Hermitian complex line bundle $\Lambda^n(T_J^{1,0}M)$, where $TM^\C= T_J^{1,0}M \oplus T_J^{0,1}M$ is the usual splitting of the complexified tangent bundle into $\pm i$-eigenvalues for $J$ and the Hermitian structure on $TM^\C$ is given by $h(X,Y)=g_J(X,Y)-i\omega(X,Y)$.
The Chern Ricci form represents thus the  cohomology class $4\pi c_1(M,\omega)$.  When $J$ is integrable $\Cric$ is  the Ricci form of the K\"ahler metric $g_J, \,Ric^{g_J}(X,JY)$.\\
S. Donaldson  proved in \cite{bib:Donaldson} that the map $s: \mathcal{J}(M,\omega)\rightarrow C^\infty(M) : J\mapsto  \scal^C=\scal^{C(\omega,J)} $ can be viewed as a moment map for the action  of the space of Hamiltonian vector fields $ham(M,\omega)=\{X\in\chi(M)\,\vert\, \iota(X)\omega=df_X \, \textrm{ for a } f_X\in C^\infty(M)\}$ on the space $\mathcal{J}(M,\omega)$ of positive compatible almost complex structures on a compact symplectic manifold $(M,\omega)$. 

V. Apostolov et T. Draghici \cite{bib:ApostolovDrag} have considered  the condition on a positive compatible almost complex structure  on a   symplectic manifold $(M,\omega)$   that its Chern Ricci form $\Cric$ be proportional to $\omega$. Della Vedova called these $J$'s {\it {special}} and 
studied  in \cite{bib:DellaVedova}  triples  $(M,\omega,J)$ with $J$ special in a homogeneous context. Further studies by Alekseevsky and Podest\`a appear in \cite{bib:AlekseevskyPodesta}. \\

\section{Variational principles} \label{section:variational principles}

In the introduction, we reviewed the variational principle built on the space $\mathcal{J}(M,\omega)$ of positive compatible almost complex structures on a  symplectic manifold $(M,\omega)$ by integrating the square of the norm of the Nijenhuis tensor.
We also mentioned that the square of the norm was the only invariant polynomial of degree $\le 2$ in $N^J$.\\

Other variational principles  on the space $\mathcal{J}(M,\omega)$ have been defined using invariant built from the curvature of a natural connection.

 D. Blair and S. Ianus \cite{bib:BlairIanus} studied the functional $\mathcal{F}_1$ which associates to a $J\in \mathcal{J}(M,\omega) $  the integral of the scalar curvature $\scal^{g_J}$ of the Levi Civita connection associated to the metric $g_J$:
$$
\mathcal{F}_1(J)=\int_M  \scal^{g_J} \frac{\omega^N}{n!}.
$$
This can be seen as  the restriction of the Hilbert functional, defined on the space of metrics on $M$, to the space of metrics built from positive compatible almost complex structures. Extrema of the Hilbert functional are the Einstein metrics.
The extrema of $\mathcal{F}_1$ are those $J\in \mathcal{J}(M,\omega) $ such that the Ricci tensor of the Levi Civita connection associated to $g_J$ is $J$-hermitian : $Ric^{g^J}(JX,JY)=Ric^{g^J}(X,Y)$. 
Examples of non K\"ahlerian triples $(M,\omega,J)$ satisfying this condition
have been given in 1990 par Davidov et Muskarov \cite{bib:DavidovMusk} on some twistor spaces over Riemannian manifolds.

J. Keller et M. Lejmi have  studied  \cite{bib:KellerLejmi}  the variational principle  whose functional is given by integrating the square of the Hermitian scalar curvature; they extend results concerning a lower bound for this integral obtained by S.K. Donaldson in the K\"ahler case, to the non K\"ahler case.\\

\section{Dimension and involutivity  of distributions associated  to $N^J$}\label{section:dimdistrib}

We now  look at properties of distributions associated to a triple $(M,\omega, J)$.

\begin{definition}\label{ImNJ} The {\it{Image distribution}} associated to an almost complex structure $J$ is 
\begin{equation}\label{def:ImNJ}
(\Image N^J)_x:=\Span\{ N^J_x(X,Y)\,\vert\, X,Y \in T_xM \}.
\end{equation}
Observe that $(\Image N^J)_x$ is invariant under $J_x$; when $J$ is compatible with a symplectic structure $\omega$ and positive, $(\Image N^J)_x$ is thus a symplectic subspace  of $T_xM$, ; one denotes by $(\Image N^J)^\perp_x$
its $\omega_x$-orthogonal complement (which coincides with its orthogonal complement relative to $g_J$ )and by $(\Image N^J)^\perp$ the corresponding distribution. \\
The {\it{Kernel distribution}} associated to  $J$ is
\begin{equation}\label{def:KerNJ}
(\Ker N^J)_x:=\{ X \in T_XM \,\vert\, N^J_x(X,Y)=0\quad \forall Y\in T_xM\}.
\end{equation}
Again  $(\Ker N^J)_x$ is invariant under $J_x$, hence for a triple $(M,\omega, J)$ it is a symplectic subspace  of $T_xM$.
\end{definition}
In general, those distributions do not have constant rank. 
Property \ref{prop:cyclicN} implies 
\begin{equation}\label{propKerIm}
(\Ker N^J)_x\subset (\Image N^J)^\perp_x.
\end{equation} 
Recall that any  almost complex structure $J$ on a $2$-dimensional  manifold is integrable.
\begin{prop}\label{prop:dim4}
For any  almost complex structure $J$ on a $4$-dimensional  manifold $M$, and for any point $x\in M$, one has 
$$\dim (\Image N^J)_x\le 2 .$$
\end{prop}
\begin{proof} Given any non vanishing $X\in T_xM$ and any $Y\in T_xM$, not belonging to $\Span \{X,J_xX\}$,
 using property \ref{prop:JNJ} and the fact that $T_xM= \Span \{X,J_xX,Y,J_xY\}$, we have
$$
(\Image N^J)_x=\Span\{N^J_x(X,Y), J_xN^J_x(X,Y)\}.
$$
\end{proof}
Remark that if $J$ is compatible with a symplectic structure $\omega$, and if $X,Y$ are in $(\Image N^J)_x^\perp$ then, by property \ref{prop:cyclicN}, $N^J(X,Y)=0$.
Hence, for such a $J$ :  
\begin{equation}\label{Ker}
(\Ker N^J)_x=\{ X\in (\Image N^J)_x^\perp \,\vert \, N^J(X,A)=0\,\, \forall A \in (\Image N^J)_x\}.
\end{equation}
\begin{prop}
For any positive compatible almost complex structure $J$ on a $2n$-dimensional symplectic manifold $(M,\omega)$, and for any point $x\in M$, one has 
$$\dim (\Image N^J)_x=2 \qquad \Leftrightarrow\qquad  \dim (\Ker N^J)_x =2n-4.$$
\end{prop}
\begin{proof}
If $\dim (\Image N^J)_x=2 $, say that $(\Image N^J)_x$ is spanned by $A$ and $JA$. Then, using properties \ref{Ker} and \ref{prop:JNJ}
$(\Ker N^J)_x=\{ X\in  <A,JA>^\perp \,\vert\,  N^J(X,A)=0  \}$.
Writing $N^J(X,A)=\alpha(X)A + \beta(X) JA$, with $\alpha$ and $\beta$ linear forms on  $<A,JA>^\perp$, property \ref{prop:JNJ} implies that $\beta(X)=\alpha (JX)$, hence $\alpha$ is not identically zero. Thus the kernel is given by
$(\Ker N^J)_x=\{ X\in  <A,JA>^\perp \,\vert\,  \alpha(X)=0 \quad \alpha\circ J (X)=0  \}$. Hence $\dim (\Ker N^J)_x =2n-4.$\\
Reciprocally, if $\dim (\Ker N^J)_x =2n-4$, then $T_xM=(\Ker N^J)_x \oplus <X,Y,JX,JY>$ and $(\Image N^J)_x$ is spanned by $N^J(X,Y)$ and $JN^J(X,Y)$
as in proposition \ref{prop:dim4}. If one of these vectors (and hence both) would vanish, $N^J$ would vanish and $(\Ker N^J)_x$ would be of dimension $2n$,
hence $\dim (\Image N^J)_x=2 $.
\end{proof}
Proposition \ref{prop:dim4}  and the $2$-dimensional case  are the only limitations to the dimension of $(\Image N^J)_x$ :
we shall give examples of symplectic manifolds endowed with positive compatible almost complex structures such that 
the image distribution is of constant rank in any possible even dimension up to the dimension of the manifold, except that in dimension $4$, the dimension of the image is at most $2$. 
The examples given here are symplectic groups endowed with invariant almost complex structures.\\

A symplectic group is a Lie group $G$ equipped with a left invariant symplectic structure $\omega$;
 this  is equivalent to the data of a non-degenerate $2$-cocycle $\Omega$ on the Lie algebra $\g$ of $G$, i.e. a non degenerate skewsymmetric bilinear map $\Omega : \g\times \g \rightarrow \R$ such that
\begin{equation}\label{cocycle}
\cyclic_{X,Y,Z} \Omega([X,Y],Z)=0.
\end{equation}
A left invariant compatible positive almost complex structure $J$ on the group $G$ then corresponds to a linear map $j: \g\rightarrow \g$ such that $\Omega(jX,jY)=\Omega(X,Y)$ and $\Omega(X,jX)>0$ for any $X\neq 0$; we call such a map a positive compatible $j$.\\
A review about symplectic groups is given in \cite{bib:Trale}. 
A nice way to produce a compact symplectic manifold is to look at the quotient of a  symplectic group by a lattice (a discrete subgroup acting cocompactly).
Malcev's criterion  says that a simply connected nilpotent Lie group admits a lattice if and only if its Lie algebra admits a basis with rational structure constants.\\

Examples of symplectic Lie algebras $(\g,\Omega)$ of dim $4$ and compatible positive $j$'s, with corresponding left invariant distributions $\Image N^J$ involutive or not and independently $(\Image N^J)^\perp$ involutive or not, are given below.\\
We assume in all cases to have a basis of $\g$  given by $\{X_1,X_2,Y_1,Y_2\}$ for which  $\Omega=\sum_{i =1}^2X^*_i\wedge Y^*_i$ and  $j(X_i)=Y_i$ for $i=1,2$.
We describe now four symplectic algebras by giving the non vanishing Lie brackets of the basis elements:
\noindent \begin{align*}
{\bf\textrm{Ex 1 :}}&[X_1,X_2]=Y_2, \, [X_1,Y_2]=Y_1 .   & \Image N^J=<X_1-Y_2,X_2+Y_1> \cr
&& \textrm{ is non involutive } \cr
&& (\Image N^J)^\perp=<X_2-Y_1,X_1+Y_2>\cr
 &&\textrm{is  non involutive } \cr
{\bf\textrm{Ex 2 :}}&[Y_1,Y_2]=X_2.    & \Image N^J=<X_2,Y_2 > \cr 
&&\textrm{is involutive } \cr
&& (\Image N^J)^\perp=<X_1,Y_1>\cr
&&\textrm{is involutive } \cr
{\bf\textrm{Ex 3 :}}&[X_1,X_2]=\half(X_2+Y_2), \, [X_1,Y_1]=Y_1    & \Image N^J=<Y_2,X_2>\cr
  &&\textrm{is non involutive } \cr
&[X_1,Y_2]=\half Y_2, \, [X_2,Y_2]=Y_1.& (\Image N^J)^\perp=<Y_1,X_1> \cr
&&\textrm{is  involutive } \cr
{\bf\textrm{Ex 4 :}}&[X_1,X_2]=-X_2+2Y_1+4Y_2,      & \Image N^J=<X_1+2X_2,Y_1+2Y_2>\cr
  &&\textrm{ is  involutive } \cr
&[X_1,Y_1]=-Y_1,\, [X_1,Y_2]=Y_1+ Y_2. & (\Image N^J)^\perp=<2X_1-X_2,2Y_1-Y_2>\cr
&&\textrm{ is non involutive } \cr
 \end{align*}
Remark that the Lie algebras in examples 1 and 2 are nilpotent and the structure constants in the given basis are integers; they are solvable in examples 3 and 4.\\

We give an example of a nilpotent symplectic Lie algebra   $\g$ of dimension $6$ with a compatible positive $j$ such that $\Image N^J$ is of dimension  $6$.
Let  $\{X_1,X_2,X_3,Y_1,Y_2,Y_3\}$ be a basis of $\g$,  set $\Omega=\sum_{i =1}^3X^*_i\wedge Y^*_i $,  $ j(X_i)=Y_i$ (for $i \le 3$) and the only non vanishing brackets of the Lie algebra are given by:
$$[X_1,X_2]=X_2+X_3,\quad [X_1,X_3]=-X_2-X_3+Y_2,\quad  [X_1,Y_2]=-Y_2+Y_3,$$
$$[X_1,Y_3]=-Y_2+Y_3,  [X_2,X_3]=Y_1.$$
Then  $N^J(X_1,X_2)=2X_3,\,\, N^J(X_1,X_3)=2X_3-Y_2,\,\,N^J(X_2,X_3)=-Y_1,$ hence $\Image N^J=\g$. Remark that the structure constants in the given basis are integer numbers, so Malcev's criterion applies and there exists a cocompact lattice in $\g$.\\

To construct symplectic groups  with invariant $J$'s, with $\Image N^J$ involutive or not (of any possible dimension) and independently $(\Image N^J)^\perp$ involutive or not, we proceed now by induction at the level of Lie algebras.
We use constructions (which are particular cases of constructions given by Medina and Revoy in \cite{bib:MedinaRevoy}) to build, from a symplectic Lie algebra $(\g,\Omega)$ of dimension $2m$, with a positive compatible $j$, and with $\g$ non perfect (i.e. the derived algebra $[\g,\g]$ not equal to $\g$) a new non perfect symplectic Lie algebra  of dimension $2m+2$ and a positive compatible  almost complex structure on this new algebra, not changing involutivity properties of the corresponding $\Image N^J$ and $(\Image N^J)^\perp$.\\

The first construction increases the dimension of $(\Image N^{J})^\perp$; at the level of groups, it is just the direct product with the $2$- dimensional  K\"ahlerian group ; at the level of algebras:
$$\widetilde{\mathfrak{g}}=\mathfrak{g}\oplus \R^{2},\quad \widetilde{\Omega}=\Omega\oplus \Omega_{std},\quad \widetilde{j}=j\oplus j_{std}$$
so that $\Image N^{\widetilde{J}}=\Image N^J, \quad (\Image N^{\widetilde{J}})^\perp=(\Image N^{J})^\perp\oplus \R^{2}.$\\

The second construction increases the dimension of $\Image N^{J}$: at the level of algebras, as vector spaces it is again a direct sum with $(\R^2,\Omega_{std},j_{std}) $ but the definition of the new brackets involves the choice of a non zero character $\xi : \mathfrak{g} \rightarrow \R$ of the Lie algebra $\g$ (which exists since $\g$ is non perfect). Precisely

$${\mathfrak{g}}'=\mathfrak{g}+ \R c+ \R d,\quad \begin{array}{l} {\Omega}'=\Omega\oplus \Omega_{std}\cr i.e.\,  \Omega'(c,d)=1\end{array},\quad \begin{array}{l} j'=j\oplus j_{std}\cr i.e.\,  j'(c)=d  \end{array},$$
with brackets in $\g'$ given by 
$$ [u,v]'=[u,v],\quad [u,c]'=-\xi(u)d,  \quad[u,d]'=0,  \quad [c,d]'=0  \quad \textrm{ for any } u,v\in \g. $$
Hence $\Image N^{J'}=\Image N^J+\R c+\R d,$ and $(\Image N^{J'})^\perp=(\Image N^{J})^\perp\subset \mathfrak{g}.$\\

We have thus proven
\begin{prop}
In  any dimension $2n$ ($n\ge 2$), and for any $k\le n$ (assuming $k<n$ if $n=2$), there exists a  symplectic Lie group of dimension $2n$ with a positive invariant compatible almost complex structure $J$ such that the distribution $\dim\Image N^{J}$ has constant rank equal to $2k$ and such that the two distributions  $\Image N^{J},(\Image N^{J})^\perp$ have any of the four involutivity properties possible (both involutive, or both non involutive, or one involutive and the other not) when $0<k<n$. 
\end{prop}

Remark that in both constructions, starting from a nilpotent Lie algebra, we get a new nilpotent Lie algebra. 
In view of Malcev's criterion, this immediately leads to
\begin{prop}
In  any dimension $2n$ ($n\ge 2$), and for any $k\le n$ (assuming $k<n$ if $n=2$), there exists a compact symplectic nilmanifold of dimension $2n$ with a positive compatible almost complex structure $J$ such that the distribution $\dim\Image N^{J}$ has constant rank equal to $2k$ and such that the two distributions  $\Image N^{J},(\Image N^{J})^\perp$ are either both involutive, either both non involutive (this possibility of course only when $0<k<n$). 
\end{prop}

\begin{definition}
We say that an almost complex structure is {\emph{ maximally non integrable }} if the image of its Nijenhuis tensor is the whole tangent space at every point.
\end{definition} 

Remark that this is an open condition in the space of compatible almost complex structures.
 From the above, a compatible  almost complex structure on a symplectic manifold of dimension $4$ is never maximally non integrable, but there are examples of invariant almost complex structures on symplectic groups which are maximally non integrable in any dimension $\ge 6$.
\section{Parallel distributions and maximally non integrable almost complex structures} \label{section:parallelimage}

\begin{definition}
A distribution $\mathcal{D}$ on a manifold $M$ endowed with a linear connection $\nabla$ is said to be {\emph{parallel}} with respect to $\nabla$  if the covariant derivative of a vector field $Y$ with values in $\mathcal{D}$, $ Y\in \Gamma(M,\mathcal{D})$, in the direction of any vector field $X$ , is again a vector field with values in $\mathcal{D}$:
$$
(\nabla_XY)(p)\in  \mathcal{D}_p\qquad \forall p\in M, \forall X\in \chi(M), \forall Y\in \Gamma(M,\mathcal{D}). 
$$ 
\end{definition} 

\begin{definition}
Given a compatible almost complex structure $J$ on a  symplectic manifold $(M,\omega)$, we shall say that it has {\emph{almost parallel Nijenhuis tensor}} if the image distribution of the Nijenhuis tensor, $ \Image N^{J} $, is parallel
with respect to the Levi Civita connection associated to the metric $g_J$. Remark that this is the case if and only if $(\Image N^{J})^\perp$ is parallel; the Levi Civita connection being torsionless, it implies that both those distributions are involutive.
\end{definition}

Recall that  an almost complex structure is maximally non integrable if the image of its Nijenhuis tensor is the whole tangent space at every point.

\begin{prop}
A positive compatible almost complex structure $J$ on a  symplectic manifold $(M,\omega)$ has almost parallel Nijenhuis tensor if and only if the manifold is locally a product of a K\"alher manifold and a symplectic manifold with a compatible almost complex structure which is maximally non integrable.
When  the manifold $M$ is connected, simply connected, and the connection is complete, we have globally
$$
(M,\omega,J)=(M_1,\omega_1,J_1)\times (M_2,\omega_2,J_2)
$$
with $(M_1,\omega_1,J_1)$ K\"ahler and $J_2$ maximally non integrable.
\end{prop}
\begin{proof}
Given a $J$ on $(M,\omega)$ which has almost parallel Nijenhuis tensor, the Riemannian manifold $(M, g_J)$ is endowed with two parallel orthogonal, supplementary involutive distributions, $TM={\mathcal{D}}_1 \oplus {\mathcal{D}}_2$ with ${\mathcal{D}}_1= \Image N^{J}$ and $ {\mathcal{D}}_2=(\Image N^{J})^\perp= {\mathcal{D}}_1^{\perp_{g_J}}$.
The result then follows from de Rham's theorem \cite{bib:KN}. 
\end{proof}

Remark that parallelism of the Nijenhuis tensor with respect to the Levi Civita connection (i.e. $\nabla^{g_J}  N^J=0$) implies almost parallel Nijenhuis tensor, but in fact one has the following stronger property.
\begin{prop}
Given a positive compatible almost complex structure $J$ on a  symplectic manifold $(M,\omega)$, if the Nijenhuis tensor $N^J$ is parallel for the Levi Civita connection associated to $g_j$, then $J$ is integrable (i.e. $\nabla^{g_J}  N^J=0 \iff N^J=0$).
\end{prop}

\begin{proof}
If $\nabla^{g_J}  N^J=0$, then $$0=\left(\nabla_X^{g_J}  N^J\right)(Y,JY)=N^J(Y,(\nabla^{g_J}_XJ)Y).$$
Using equation \eqref{eq:NnablaJ} ( $2\omega\left( (\nabla_X^{g_J}J)Y,Z\right)=\omega(N^J(Y,Z),JX)$) 
one has
\begin{eqnarray*}
2g_J((\nabla_X^{g_J}J)Y,(\nabla_X^{g_J}J)Y)&=&2\omega\left(\left(\nabla_X^{g_J}J\right)Y,J\left(\nabla_X^{g_J}J\right)Y\right)\\
&=&\omega\left(N^J\left(Y,J(\nabla_X^{g_J}J)Y\right),JX\right)\\
&=&-\omega(JN^J(Y,(\nabla_X^{g_J}J)Y),JX)=-\omega(N^J(Y,(\nabla_X^{g_J}J)Y),X)\\
&=& 0
\end{eqnarray*}
hence $\nabla_X^{g_J}J=0$ and so $N^J=0$.
\end{proof}
  \section{ Twistor space on the hyperbolic spaces} \label{section:twistor} 
  
We consider the manifold $\mathcal{T}$ which is the twistor space over the $2n$-dimensional hyperbolic space \begin{eqnarray*}
H_{2n}&=&\{(x^0,x^1,\ldots,x^{2n})\in \R^{1,2n}\,\vert\, -(x^0)^2+\sum_{i=1}^{2n}(x^i)^2=-1 \textrm{ and } x^0>0\}\\
&=& SO_0(1,2n)/SO(2n) \textrm{ where  }SO_{2n}=\{ \left( \begin{matrix}
1 & 0\\ 0 & B \end{matrix}\right)\in SO_0(1,2n)\};\end{eqnarray*}
it fibers over $H_{2n}$ and the fiber over the point $x$ consists of all almost complex structures on $T_xH_{2n}$ which are compatible with the Riemannian metric $g$ on $H_{2n}$ (induced by the Lorentzian scalar product on $\R^{1,2n}=\R e_0\oplus\R^{2n}$) and compatible with the orientation, i.e. such that there is an oriented orthonormal basis of $T_xH_{2n}$ in which it writes $j_0=\left( \begin{matrix}
0 & -I_n\\ I_n & 0 \end{matrix}\right)$. The fiber bundle $\mathcal{F}$ of oriented orthonormal frames on $H_{2n}$ can be identified with $SO_0(1,2n)\rightarrow H_{2n}$: the tangent space $T_xH_{2n}$ is identified with $\{y\in \R^{1,2n}\,\vert\, x^0y^0-\sum_{i=1}^{2n}x^iy^i=0\}$, and the element $A\in SO_0(1,2n)$ is identified with the frame
$\tilde{A}:=A\vert_{\R^{2n}}:\R^{2n}\rightarrow T_xH_{2n}$ at the point $x:=Ae_0$.

The group $SO_0(1,2n)$ acts on $\mathcal{T}$: for $j\in \End(T_xH_{2n}),A\cdot j:=A\circ j\circ A^{-1}\in \End(T_{Ax} H_{2n})$; this action is transitive and we have
$$\mathcal{T}=SO_0(1,2n)/U(n) \stackrel{\pi}{\longrightarrow} H_{2n}\textrm{ where  } U(n)=\{ C\in SO(2n)\,\vert\, Cj_0=j_0 C\}.$$
Let us denote by $\pi_1:\mathcal{F}=SO_0(1,2n) \rightarrow \mathcal{T}=SO_0(1,2n)/U(n)$ the natural projection.
The base point $t_0:=\pi_1(\Id)$  represents the almost complex structure on $T_{e_0}H_{2n}= \R^{2n}$ given by $j_0$. The tangent space $T_{t_0} \mathcal{T}$ is given by
\begin{equation}
T_{t_0} \mathcal{T}=\pi_{1*_{\Id}}\left\{ A\in so(1,2n)\,\vert\,A=\left( \begin{matrix}
0 & ^tu\\ u & B \end{matrix}\right) \textrm{ with } u\in \R^{2n}, B\in so(2n), Bj_0=-j_0 B \right\}
\end{equation}
The tangent bundle of the twistor space splits as the direct sum of the vertical bundle and the horizontal bundle 
$$T\mathcal{T}=\mathcal{V}+\mathcal{H}$$
 where $\mathcal{V}_t:=\Ker \pi_{*t}$ and $\mathcal{H}_{t=\pi_1(f)}$ is the projection  $\mathcal{H}_{t=\pi_1(f)}=\pi_{1*_f}H_f(\mathcal{F})$ in $T_t\mathcal{T}$,
of the horizontal space  $H_f(\mathcal{F}):=\Ker \alpha_f \subset T_f\mathcal{F}$ defined by the Levi Civita connection on $(H_{2n},g)$ viewed as a $so(2n)$-valued $1$-form on $\mathcal{F}$.\\
It is a $SO_0(1,2n)$-left invariant decomposition which at the base point is given by

\begin{eqnarray} 
T_{t_0} \mathcal{T}&=&\mathcal{V}_{t_0}\oplus \mathcal{H}_{t_0}\quad\textrm{ with }\quad \mathcal{V}_{t_0}= \pi_{1*_{\Id}}\mathfrak{q} \quad\textrm{ and }\mathcal{H}_{t_0}=\pi_{1*_{\Id}}\mathfrak{p}\quad\textrm{ where }\\
\mathfrak{q}:&=&\left\{A\in so(1,2n)\,\vert\,A=\left( \begin{matrix}
0 & 0\\ 0 & B \end{matrix}\right) \textrm{ with } B\in so(2n), Bj_0=-j_0 B \right\}, \label{eq:defq}\\
\mathfrak{p}:&=&\left\{
A\in so(1,2n)\,\vert\,A=\left( \begin{matrix}
0 & ^tu\\ u & 0 \end{matrix}\right) \textrm{ with }u\in \R^{2n} \right\}.\label{eq:defp}
\end{eqnarray} 
The twistor space  can be viewed as the orbit of $j'_0:=\left( \begin{matrix}
0 & 0\\ 0 & j_0 \end{matrix}\right) \in so(1,2n)$ under the adjoint action; as such it carries the left invariant Kirillov-Kostant-Souriau symplectic structure $\omega$ given at the base point by:
$$
\omega_{t_0}(\pi_{1*_{\Id}}A,\pi_{1*_{\Id}}A'):=-\Tr j'_0[A,A']
$$ so that 
$$\omega_{t_0}(\pi_{1*_{\Id}}\left( \begin{matrix}
0 & ^tu\\ u & B \end{matrix}\right),\pi_{1*_{\Id}}\left( \begin{matrix}
0 & ^tv\\ v & B' \end{matrix}\right))=2 u\cdot j_0 v -\Tr j_0 [B,B'].
$$
As a twistor space, it carries two natural almost complex structures $J^{\pm}$; those are the left invariant structures given on the $T_{t_0} \mathcal{T}$ by
\begin{equation}
J^{\pm}_{t_0}\left( \pi_{1*_{\Id}}\left(\begin{matrix}
0 & ^tu\\ u & B \end{matrix}\right)\right)=\left(\pi_{1*_{\Id}}\left( \begin{matrix}
0 & \pm ^tj_0u\\ \pm j_0u & j_0B \end{matrix}\right)\right).
\end{equation}
Observe that they are both compatible with $\omega$ and that $J^-$ is positive whereas $J^+$ is not.
\begin{lem}\label{lem:NJhomog}\cite{bib:DellaVedova}
Let $M$ be  a $G$-homogeneous  manifold  endowed with a $G$-invariant symplectic form $\omega$ and a $G$-invariant compatible almost complex structure $J$. We denote by $A^*$ the fundamental vector field on $M$ defined by an element $A\in \g$ (i.e. $A^*_x=\frac{d}{dt} \exp -tA\cdot x\vert_{t=0}$), we choose a base point $x_0$,  
and we choose a linear map  $J_0 :\g\rightarrow \g$  such that $J_{x_0}A^*=(J_0A)^*_{x_0}$ for all $A\in \g$. Then
\begin{equation}\label{eq:NJinv}
N^J_{x_0}(A^*,B^*)=\left(-N^{J_0}(A,B)\right)^*_{x_0}
\end{equation}
\begin{equation}
\textrm{ with } N^{J_0}(A,B):=[J_0A,J_0B] -J_0[J_0A,B]-J_0[A,J_0B]-[A,B] \quad \forall A,B \in\g.
\end{equation}
\end{lem}
\begin{proof}
The metric $g_J$ associated to $\omega$ and $J$ is $G$-invariant, hence $L_{A^*}g=0$ and
the Levi Civita connetion $\nabla=\nabla^{g_J}$ is given by
\begin{equation}\label{LChomog}
2g_J(\nabla_{A^*}B^*,C^*)=(g_J([A,B]^*,C^*) +g_J([A,C]^*,B^*) + g_J([B,C]^*,A^*). 
\end{equation}
Using the fact that $\nabla$ has no torsion, the Nijenhuis tensor is  given by
\begin{eqnarray*}
N^J(X,Y)&=&(\nabla_{JX}J)Y-(\nabla_{JY}J)X-J(\nabla_{X}J)Y+J(\nabla_{Y}J)X \quad  \forall X,Y\in \chi(M)\\
 &=& 2J(\nabla_{Y}J)X-2J(\nabla_{X}J)Y \quad \textrm {using equation}\eqref{eq:nablaJN}\\
 &=& 2J \nabla_{JX}Y +2J[Y,JX]+2\nabla_YX-2J \nabla_{JY}X -2J[X,JY] -2\nabla_XY.
 \end{eqnarray*}
 Hence, using the invariance of $J$:
 \begin{equation}
 N^J(A^*,B^*)=2J \nabla_{JA^*}B^*-2J \nabla_{JB^*}A^* +2[A,B]^* \quad \forall A,B \in \g
 \end{equation}
 so that, at the base point $x_0$, one has
  \begin{equation}
 N^J_{x_0}(A^*,B^*)=2J_{x_0} \nabla_{(J_0A)^*}B^*-2J_{x_0} \nabla_{(J_0B)^*}A^* +2[A,B]^* \quad \forall A,B \in \g
 \end{equation}
Equation \eqref{LChomog} yields 
{\scriptsize{\begin{eqnarray*}
 \left( g_{J}(N^J(A^*,B^*),C^*)\right)_{x_0}&=&-2g_{Jx_0}( \nabla_{(J_0A)^*}B^*,(J_0C)^*)
 +2g_{Jx_0}( \nabla_{(J_0B)^*}A^*,(J_0C)^*) 
+g_{Jx_0}(2[A,B]^*,C^*)\\
 &=&-\left(g_J([J_0A,B]^*, (J_0C)^*)-g_J([J_0A,J_0C]^*, B^*)-g_J([B,J_0C]^*, (J_0A)^*)\right.\\
 &&+g_J([J_0B,A]^*, (J_0C)^*)+g_J([J_0B,J_0C]^*, A^*)+g_J([A,J_0C]^*,(J_0 B)^*)\\
 &&\left.+g_{J}(2[A,B]^*,C^*)\right)_{x_0}\\
 &=& \left(g_{J}\left((-[J_0A,J_0B]+J_0[J_0A,B]+J_0[A,J_0B]+[A,B])^*,C^*\right)\right)_{x_0}\\
 &&+\left(g_J([J_0A,J_0B]^*,C^*)-g_J([J_0A,J_0C]^*, B^*)+g_J([J_0B,J_0C]^*, A^*)\right)_{x_0}\\
 &&+\left(g_J([A,J_0C]^*,(J_0 B)^*)+g_{J}([A,B]^*,C^*)-g_J([B,J_0C]^*, (J_0A)^*)\right)_{x_0}\\
&=& \left(g_{J}\left(-\left(N^{J_0}(A,B)\right)^*,C^*\right)\right)_{x_0}\\
&& \cyclic_{A,B,C}\omega_{x_0}([(J_0A)^*,(J_0B)^*],(J_0C)^*) \cyclic_{A,B,J_0C}\omega_{x_0}([A^*,B^*],(J_0C)^*).
\end{eqnarray*}}}
The last  line vanishes because $\omega$ is closed; indeed $\omega$ being invariant, $L_{A^*}\omega=0$,and
 $$d\omega(A^*,B^*,C^*)=\cyclic_{A,B,C} \omega([A,B]^*,C^*).$$
Hence $\left( g_{J}(N^J(A^*,B^*),C^*)\right)_{x_0}=\left(g_{J}\left(-\left(N^{J_0}(A,B)\right)^*,C^*\right)\right)_{x_0}$ for all $A,B,C \in \g$, and this proves equation \ref{eq:NJinv}.
\end{proof}
In \cite{bib:DellaVedova} the compatible almost complex structure is supposed to be positive to have a reductive homogeneous space.

\begin{prop} For the twistor space $\mathcal{T}$  over the $2n$-dimensional hyperbolic space,
endowed with its Kirillov-Kostant-Souriau symplectic form, the   almost complex structures $J^+$ is compatible and integrable, but not positive, and the almost complex structure $J^-$ is positive, compatible, and maximally non integrable.
\end{prop}
\begin{proof}
It remains to show that $N^{J^+}=0$ and that $\Image N^{J^-}_t=T_t\mathcal{T}$.
We choose a base point  $x_0=t_0=\pi_1(\Id)$. By invariance under the action of $G=SO_0(1,2n)$, it is enough to prove the properties
for $N^{J^\pm}_{t_0}$. Observe that $\pi_{1*_{\Id}}A=-A^*_{t_0}$ for any $A\in\g=\mathfrak{so}(1,2n)$. This Lie algebra splits as
$\mathfrak{so}(1,2n)=\mathfrak{u}(n)\oplus \mathfrak{q}\oplus \mathfrak{p}$ and $A^*_{t_0}=0$ iff $A\in \mathfrak{u}(n)=
\left\{A\in \mathfrak{so}(1,2n)\,\vert\,A=\left( \begin{matrix}
0 & 0\\ 0 & B \end{matrix}\right) \textrm{ with } B\in \mathfrak{so}(2n) \textrm{ and } Bj_0=j_0 B \right\}$.\\

We can choose the map $J^\pm_0:\g\rightarrow\g$ such that $J_{t_0}A^*=(J_0A)^*_{t_0}$
 given by
 $$
 J^\pm_0\left( \begin{matrix}
0 & ^tu\\ u & B \end{matrix}\right)=\left( \begin{matrix}
0 & \pm ^t(j_0u)\\ \pm j_0u & j_0B \end{matrix}\right).$$ 
Using Lemma \ref{lem:NJhomog} and equation \eqref{eq:NJinv}, the image of $N^{J^\pm}_{t_0}$ is  the space of fundamental vector fields at $t_0$ corresponding to the image of $N^{J^\pm_0}$. It is thus enough to compute
$$N^{J^\pm_0}(A,A')=[J^\pm_0A,J^\pm_0A'] -J^\pm_0[J^\pm_0A,A']-J^\pm_0[A,J^\pm_0A']-[A,A']\quad\forall A,A' \in\mathfrak{so}(1,2n).$$
If $Bj_0=-\epsilon_B j_0B$ and $Cj_0=-\epsilon_C j_0C$ with 
$\epsilon_B= \pm 1$ and $\epsilon_C=\pm 1$ ($+1$ corresponding to an element in $\mathfrak{q}$ and $-1$ corresponding to an element in $\mathfrak{u}(n)$), a  straightforward computation gives
$$
N^{J^\pm_0}\left( \left( \begin{matrix}
0 & ^tu\\ u & B \end{matrix}\right),\left( \begin{matrix}
0 & ^tv\\ v & C \end{matrix}\right)     \right)=\left( \begin{matrix}
0 & ^tw\\ w & D \end{matrix}\right)
$$
where $D:=(\pm 1-1)\left(u \otimes ^tv-v\otimes ^tu-j_0u\otimes ^t(j_0v)-j_0v\otimes ^t(j_0u)\right)$\\ and 
$w:=(\pm 1-1)\left((\epsilon_A+1)Av-(\epsilon_B+1)Bv\right)$.\\
Clearly, $N^{J^+_0}$ vanishes identically, $N^{J^-_0}$ vanishes on $\mathfrak{so}(2n)\times \mathfrak{so}(2n)$ but $N^{J^-_0}$ restricted to $\mathfrak{p}\times \mathfrak{q}$ has for image the whole of  $\mathfrak{p}$ ($w=-4Av$ can be any vector in $\R^{2n}$) and $N^{J^-_0}$ restricted to $\mathfrak{p}\times \mathfrak{p}$ has for image the whole of  $\mathfrak{q}$ since any matrix
$B\in  \mathfrak{so}(2n)$ such that $Bj_0=-j_0 B$ is in the span of the matrices of the form
$\left(u \otimes ^tv-v\otimes ^tu-j_0u\otimes ^t(j_0v)+-j_0v\otimes ^t(j_0u)\right)$ with $u,v\in\R^{2n}$. Hence the image of $N^{J^-_0}$ is $\mathfrak{p}\oplus \mathfrak{q}$ and the 
corresponding fundamental vector fields at $t_0$ give the whole tangent space $T_{t_0}\mathcal{T}$.
\end{proof}

\end{document}